\documentclass[11pt, a4paper]{amsart}
\newtheorem{prop}{Proposition}[section]
\newtheorem{lema}[prop]{Lemma}
\newtheorem{teo}[prop]{Theorem}
\newtheorem{corolario}[prop]{Corollary}

\newtheorem{remark}[prop]{\sc Remark}

\usepackage{hyperref}

\newcommand{\uno}{1\!\!1}

\title[Valuations on Banach lattices]{Valuations on Banach lattices}

\author{Pedro Tradacete}

\address{Mathematics Department\\ Universidad Carlos III de Madrid \\  28911 Legan\'es (Madrid). Spain.}
\email{ptradace@math.uc3m.es }

\thanks{Support of Spanish MINECO under grants MTM2016-76808-P and MTM2016-75196-P is gratefully acknowledged.}

\author{Ignacio Villanueva}

\address{Departamento de An\'alisis Matem\'atico \\
Facultad de Matem\'aticas \\ Universidad Complutense de Madrid \\
Madrid 28040}

\email{ignaciov@mat.ucm.es}

\thanks{Partially supported by grants MTM2014-54240-P, funded by MINECO and QUITEMAD+-CM, Reference: S2013/ICE-2801, funded by Comunidad de Madrid}

\begin{document}

\begin{abstract}
We provide a general framework for the study of valuations on Banach lattices. This complements and expands several recent works about valuations on function spaces, including $L_p(\mu)$, Orlicz spaces and spaces $C(K)$ of continuous functions on a compact Hausdorff space. In particular, we study decomposition properties, boundedness and integral representation of continuous valuations.
\end{abstract}



\maketitle

\section{Introduction}

A valuation is a function $V$, defined on a given class of sets $\mathcal S$, which satisfies that for every $A,B\in\mathcal S$
$$
V(A\cup B)+V(A\cap B)=V(A)+V(B),
$$
whenevery $A\cup B$ and $A\cap B$ also belong to $\mathcal S$.
Valuations are a generalization of the notion of measure, and have become a relevant area of study in convex geometry \cite{Alesker:14, ABS, Ber, Lu1, Lu2}.  

The natural identification between convex bodies and their support functions, or between star bodies and their radial functions induces the notion of valuation on a function space as a function $V$, defined on a given class of functions $\mathcal F$, satisfying that,  for every $f,g\in\mathcal F$
$$
V(f\vee g)+V(f\wedge g)=V(f)+V(g),
$$
whenever  $f\vee g$ and $f\wedge g$ also belong to $\mathcal F$. 

Valuations on function spaces have become object of intense study in recent years \cite{CLM,CLP,Lu}. As in the case of valuations on classes of sets, also in the case of valuations on function spaces, our interest mainly focuses in the study of {\em continuous} valuations, where continuity is referred to some natural topology in the corresponding class. 

Many of the function valuations already studied correspond to classes  $\mathcal F$ closed under the operations $\vee$ and $\wedge$ and with a norm that fits well with these operations. Adding the natural linear structure of function spaces we essentially run into the definition of {\em Banach lattice}. Natural examples of Banach lattices of functions are $L_p$, $C(K)$ and Orlicz spaces. Valuations on these spaces have been the object of research in \cite{Klain96}, \cite{Klain97}, and more recently in \cite{Ts}, \cite{Ko}, \cite{Vi}, \cite{TraVi}.

In this work, we provide a general framework for the study of valuations on Banach lattices and we obtain several results which apply to a vast generality of lattices. Our aim is to show that functional analytic methods, and in particular, techniques from Banach lattice theory, can provide new insight to problems arising in metric geometry. To achieve this we profit from two main sources. On the one hand, the central role played by $C(K)$ spaces among Banach lattices,	 allows us to apply and generalize several results from \cite{TraVi} in this broader context. On the other hand, in Functional Analysis, valuations on lattices of functions had been already considered, in the study of {\em orthogonally additive} applications  \cite{Batt, FK:Lp, KZPS, MM}. We recover and extend some of the techniques and results from those papers, most notably a quite general integral representation result \cite[Theorem 3.2]{DO69}. 

With the help of these previous work, we will prove new structural results for valuations in Banach lattices and recover, with bigger generality and simpler proofs, existing results from the recent literature. 

In Section \ref{s:bbs} we prove our first two estructural results on valuations on Banach lattices. 
The first one, Proposition \ref{p:bbs}, shows that valuations on Banach lattices are bounded on order bounded sets. The second one, Theorem \ref{t:jordan} shows that continuous valuations on Banach lattices decompose as the difference of {\em positive} continuous valuations. 

Both of the proofs are heavily based on the corresponding versions for $C(K)$ spaces and the local representation of Banach lattices as $C(K)$ spaces. The $C(K)$ versions of the results were already known for case of $K$ compact and  metrizable \cite{TraVi}. However, we need to extend them first to the case of a general compact space $K$.

When restricting our attention to a still quite comprehensive class of Banach lattices, containing in particular the $L_p$ spaces ($1\leq p <\infty$) over $\mathbb R^n$ or $\mathcal S^{n-1}$, we can strengthen the first of these results: 

\begin{teo}\label{normboundedsets}
Let $E$ be a Banach lattice of measurable functions over a purely non-atomic $\sigma$-finite measure space $(\Omega,\Sigma,\mu)$. If $E$ satisfies a lower $q$-estimate for some $q<\infty$, then every valuation on $E$ which is continuous at $0$ is bounded on norm bounded sets.
\end{teo}

The proof of this result is given in Section \ref{s:bnbs}, as well as the atomic counterpart. 

In Section \ref{s:oa} we show that for a very general class of Banach lattices, valuations coincide with orthogonally additive applications. Probably, the most significant exception are $C(K)$ spaces, with $K$ connected. 

Using the coincidence of orthogonally additive applicatons with valuations, in Section \ref{s:DO} we start presenting the main result given in \cite[Theorem 3.2]{DO69} with the current terminology. The decomposition of a valuation in positive and negative part proved in the previous sections allows us to extend  \cite[Theorem 3.2]{DO69} to  Banach lattices supported on spaces of $\sigma$-finite measure. The result is the following:

\begin{teo}\label{t:DOsigmafinitas}
Let $X$ be an order continuous Banach lattice represented as a function space over a $\sigma$-finite measure space $(\Omega, \Sigma, \mu)$, and let $V:X\longrightarrow \mathbb R$ be a continuous valuation. Then,  there exists a strong Carath\'eodory function  $K:\mathbb R \times \Omega \longrightarrow \mathbb R$ such that, for every $f\in X$, $$V(f)=\int_\Omega K(f(t),t)d\mu(t).$$

Conversely, if $K:\mathbb R \times \Omega \longrightarrow \mathbb R$ is a strong Carath\'eodory function such that, for every $f\in X$, $\int_\Omega K(f(t),t)d\mu(t)<\infty$, then the function $V:X\longrightarrow \mathbb R$ defined by  $$V(f)=\int_\Omega K(f(t),t)d\mu(t)$$ is a continuous valuation. 
\end{teo}

See Section \ref{s:DO} for the explanation of the terminology appearing in this result.

Note that this provides a fairly general integral representation theorem, in the line of the main result of \cite{TraVi, TraVi2}. It should be noted though that the results of \cite{TraVi, TraVi2} apply only to $C(K)$ spaces, which are not order continuous, and are therefore independent from Theorem \ref{t:DOsigmafinitas}. Moreover, the techniques in the proofs are also quite different. 

In the rest of Section \ref{s:DO} we extract several consequences of Theorem \ref{t:DOsigmafinitas}.  First, we show an $L_1$ factorization result for valuations (Theorem \ref{t:L1factorization}). Second, we consider the case of valuations with  enough invariance. In that context, Theorem \ref{t:DOsigmafinitas} admits a much simpler form, Corollary \ref{c:representacion}. This last result includes as particular cases the main results of \cite{Ts} and \cite{Ko}.

\section{Preliminaries and notation} 

Recall a Banach lattice $E$ is a real Banach space endowed with a partial order such that 
\begin{enumerate}
\item $x+z\leq y+z$ whenever $x,y,z\in E$ satisfy $x\leq y$,
\item $\lambda x\geq0$ whenever $x\geq 0$ and $\lambda\in\mathbb R_+$, 
\item for every $x,y\in E$ there exist a least upper bound $x\vee y$ and a greatest lower bound $x\wedge y$ in $E$,
\item $\|x\|\leq\|y\|$ whenever $|x|\leq |y|$ (where for $x\in E$, we define $|x|=x\vee(-x)$).
\end{enumerate}

The theory of Banach lattices provides an abstract context which allows us to deal, in a unified way, with many of the function spaces that arise in real analysis. In particular, $L_p(\mu)$ spaces, spaces $C(K)$ of continuous functions on a compact Hausdorff space, or any space with an unconditional basis admit a natural Banach lattice structure. 

For general background on Banach lattices and any unexplained terminology we refer to the monographs \cite{LT2}, \cite{MN}. For the convenience of the reader, let us recall that a Banach lattice $E$ is order continuous if for every downward directed set $(x_\alpha)_{\alpha\in A}$ with $\bigwedge_{\alpha\in A} x_\alpha=0$, $\lim_{\alpha}\|x_\alpha\|=0$.
 
Given a Banach lattice $E$ and $f\in E_+$, the interval $[-f, f]$ is defined as $$[-f, f]:=\{g\in E \mbox{ such that } |g|\leq f\}.$$

Also, given two Banach lattices $E, F$, we say that the mapping $j:E\longrightarrow F$ is a lattice homomorphism if it is linear and, for every $f, g\in E$, $j(f\vee g)=j(f)\vee j(g)$ and  $j(f\wedge g)=j(f)\wedge j(g)$. With this notation, we state next the following well known result which provides a representation of intervals in Banach lattices as unit balls of $C(K)$ spaces.

\begin{lema}\label{l:local}
Let $E$ be a Banach lattice and $f\in E_+$. There exist a compact Hausdorff space $K$ and an injective lattice homomorphism $j:C(K)\rightarrow E$ which maps the unit ball in $C(K)$ onto the interval $[-f,f]$ in $E$.
\end{lema}

\begin{proof}
It follows, for instance, from \cite[Theorems 4.21 and 4.29]{AB}.
\end{proof}

We finish this  section with the definition of valuation that will be used along the paper. More specific definitions are included when needed. 
We add condition (3) for simplicity in the explanations. It is obvious how to adapt the reasonings if we remove that condition.

Given a Banach lattice $E$, a mapping $V:E\rightarrow \mathbb R$ is a continuous valuation if and only if for $f_n, f, g\in E$ it satisfies
\begin{enumerate}

\item $V(f) + V(g) = V(f\vee g) + V(f\wedge g)$.

\item $|V(f_n)-V(f)|\rightarrow0$ whenever $\|f_n-f\|_E\rightarrow0$,

\item $V(0)=0$
\end{enumerate}

Note that every linear functional $x^*\in E^*$ clearly defines a valuation on $E$. Moreover, if $V:E\rightarrow \mathbb R$ is a valuation and $T:F\rightarrow E$ is a lattice or anti-lattice mapping then $V\circ T$ is a valuation on $F$. Recall that a lattice (respectively, anti-lattice) mapping satisfies $T(f\vee g)=Tf\vee Tg$ and $T(f\wedge g)=Tf\wedge Tg$ (resp., $T(f\vee g)=Tf\wedge Tg$ and $T(f\wedge g)=Tf\vee Tg$).

\begin{remark}
Often  valuations are  only defined on the cone of positive functions $E_+$. In general, we can always extend every mapping $V:E_+\rightarrow X$ satisfying (1)-(3) to a valuation $\hat{V}:E\rightarrow X$, for instance setting $\hat{V}(f)=V(f_+)-V(f_-)$, where $f_+=f\vee 0$ and $f_-=(-f)\vee 0$, or setting $\hat{V}(f)=V(f\vee 0)$. 

Conversely, valuations defined on the cone of positive (or negative) functions suffice to understand valuations in general, through the following simple reasoning: Given $V:E\longrightarrow \mathbb R$, define $V_1(f)=V(f\vee 0)$ and $V_2(f)=V(f\wedge 0)$. Then, it is easily seen that $V_1, V_2$ are valuations and $V=V_1+V_2$. 

Note that linear combination of valuations on $E$ is again a valuation on $E$. Thus, valuations are naturally endowed with a linear structure. 
\end{remark}

\section{Order-boundedness and decomposition of valuations}\label{s:bbs}

In this section we prove our first two general structural results on valuations on Banach lattices. Namely, we prove that every continuous valuations on a Banach lattice $E$ can be decomposed as the difference of positive continuous valuations. And we also prove that every such valuation is bounded on orded bounded sets of $E$. Theorem \ref{t:jordan} will be very useful in the next sections. 

Both results use the standard local representation of Banach lattices as continuous functions of some compact Hausdorff space. Before we state this, we introduce some notation.

We say that a valuation $V:E\longrightarrow \mathbb R$ is bounded on norm-bounded sets if for every $\lambda>0$ there exists a real number $C>0$ such that
$$
\sup_{\|f\|_E\leq\lambda}|V(f)|\leq C.
$$
Similarly, $V:E\longrightarrow \mathbb R$ is bounded on order-bounded sets if for every $g\in E_+$ there exists $C>0$ such that
$$
\sup_{|f|\leq g}|V(f)|\leq C.
$$

Clearly, if $V:E\longrightarrow\mathbb R$ is bounded on norm-bounded sets, then it is bounded on order-bounded sets. It is clear that when $E$ is a space of the form $C(K)$ both notions coincide, and we actually have:

\begin{lema}\label{l:bbs}
Let $K$ be a compact Hausdorff space. Every continuous valuation $V:C(K)\longrightarrow \mathbb R$ is bounded on bounded sets.
\end{lema}

\begin{proof} 
The proof proceeds exactly as in \cite[Lemma 3.1]{TraVi}.
\end{proof}

 Lemmas \ref{l:local} and \ref{l:bbs} immediately imply our first simple boundedness result.

\begin{prop}\label{p:bbs}
Every continuous valuation on a Banach lattice is bounded on order bounded sets.
\end{prop}

For the next result we need to extend \cite[Lemma 3.3]{TraVi} to the case of a non-metrizable compact space $K$. In \cite{TraVi}, this result is given for  valuations on $C(\mathcal S^{n-1})$. The proof extends to valuations on  $C(K)$, with $K$ compact metrizable. However, this is not enough for our  purposes, since the compact $K$  appearing in Lemma \ref{l:local} need not be metrizable in general. 

In order to overcome this, we follow a similar approach to that of \cite{FK:69} involving a standard use of uniformities, as can be found in \cite[Chapter 8]{E}: Let $\Delta$ denote the diagonal in $K\times K$. Let $\mathcal U$ be the uniformity consisting of the open symmetric neighborhoods of $\Delta$ in $K\times K$. This family of sets satisfies the following key properties:
\begin{enumerate}
\item If  $\omega_1,\omega_2\in\mathcal U$, then $\omega_1\cap \omega_2\in\mathcal U$.
\item For every $\omega\in U$ there is $\omega'\in U$ such that $2\omega'\subset \omega$ (where $2\omega'=\{(s,t)\in K\times K: (s,r),(r,t)\in\omega'\textrm{ for some }r\in K\}$).
\item $\Delta=\bigcap_{\omega\in \mathcal U}\omega$.
\end{enumerate}
In particular, every $f\in C(K)$ is uniformly continuous with respect to $\mathcal U$, that is for every $\epsilon>0$ there is $\omega\in \mathcal U$ such that $|f(s)-f(t)|< \epsilon$ whenever $(s,t)\in\omega$. 

Given $A\subset K$, and $\omega\in\mathcal U$, let
$$
A_\omega=\{k\in K:(a,k)\in \omega\textrm{ for some }a\in A\}.
$$
The $\omega$-rim around $A$ is the set
$$
R(A,\omega)=A^c\cap A_\omega.
$$
Note that, for every closed set $A\subset K$ and $\omega\in\mathcal U$, $R(A,\omega)$ is an open set. 

Our  next result generalizes \cite[Lemma 3.3]{TraVi}  to general compact spaces $K$. Before, let us recall the following notation: given $f\in C(K)$ and $A\subset K$ we will denote $f\prec A$ whenever 
$$
supp(f)=\overline{\{t\in K:f(t)\neq0\}}\subset A.
$$

\begin{lema}\label{rims}
Let $V:C(K)\rightarrow \mathbb R$ be a valuation. Let $A\subset K$ be closed and $\lambda\in \mathbb R^+$.
$$
\lim_{\omega\rightarrow \Delta} \sup\{|V(f)|: \, f\prec R(A,\omega), \, \|f\|_\infty\leq \lambda\}=0.
$$
\end{lema}

\begin{proof}
We reason by contradiction. Suppose the result is not true. Then there exist a closed set $A\subset K$, $\lambda \in \mathbb R^+$, $\epsilon>0$, a decreasing family $(\omega_j)_{j\in J}\subset \mathcal U$ with $\bigcap_j \omega_j=\Delta$ and a family $(f_j)_{j\in J}\subset C(K)$ such that for every $j\in J$:

\begin{itemize}
\item  $f_j\prec R(A,\omega_j)$,
\item $\|f_j\|_\infty\leq \lambda$,
\item $|V(f_j)|\geq \epsilon.$
\end{itemize}

Therefore, there exists an infinite subset $I\subset J$ such that either  $V(f_i) >\epsilon$ for every $i\in I$ or $V(f_i) <-\epsilon$ for every $i\in I$. We assume without loss of generality that $V(f_i) >\epsilon$ for every $i\in I$. The case $V(f_i) <-\epsilon$ is totally analogous.

Consider $f_1$. Using the continuity of $V$ at $f_1$, we get the existence of $\delta>0$ such that for every  $g\in C(K)$ with $\|f_1-g\|_\infty<\delta$,
$$
|V(f_1)-V(g)|\leq \frac{\epsilon}{2}.
$$

Since $f_1$ is uniformly continuous and $f_1(t)=0$ for every $t\in A$, there exists $\omega'_1$ such that, for every $t\in A_{\omega'_1}$, $|f_1(t)|<\delta$. Let $\omega^{''}_1\in\mathcal U$ such that $2\omega^{''}_1\subset \omega'_1$, and consider the closed sets
$$
C_1=\overline{A_{\omega^{''}_1}}
$$
and
$$
C_2=f_1^{-1}\left([-\lambda, \lambda]\backslash(-\delta,\delta)\right).
$$
Note $C_1\cap C_2=\emptyset$. Indeed, if $t\in C_1$, then there is $t_1\in A_{\omega^{''}_1}$ such that $(t,t_1)\in \omega^{''}_1$ (cf. \cite[Corollary 8.1.4]{E}). Since $2\omega^{''}_1\subset \omega'_1$, it follows that $t\in A_{\omega'_1}$, so $|f(t)|<\delta$, and in particular $t\notin C_2$, as claimed.

By Urysohn's Lemma, we can consider a continuous function $\psi_1$ with $\psi_{1|_{C_1}}=0$, $\psi_{1|_{C_2}}=1$ and  $0\leq \psi_1(t)\leq 1$ for every $t\in K$. We consider now the function $\psi_1 f_1\in C(K)$. On the one hand, $\|f_1-\psi_1 f_1\|_\infty\leq \delta$ and, therefore, $$|V(\psi_1 f_1)|\geq \left||V(f_1)|-|V(f_1)-V(\psi_1 f_1)|\right|> \epsilon-\frac{\epsilon}{2}=\frac{\epsilon}{2}.$$

On the other hand,  $\psi_1 f_1\prec R(A,\omega_1)\backslash \overline{A_{\omega^{''}_1}}$. Now, we can choose $\omega_{j_2}\in \mathcal U$ such that $\omega_{j_2}\subset�\omega^{''}_1$ and we can reason similarly as above with the function $f_{j_2}$.

Inductively, we construct a sequence of functions $(\psi_k f_{j_k})_{k\in \mathbb N}\subset C(K)$ with disjoint supports such that $V(\psi_k f_{j_k})>\frac{\epsilon}{2}$. Noting that
$$
V\bigg(\bigvee_k \psi_k f_{j_k}\bigg)=\sum_k V(\psi_k f_{j_k}),
$$
and that
$$
\bigg\|\bigvee_k \psi_k f_{j_k}\bigg\|_\infty\leq\lambda,
$$
we get a contradiction with the fact that $V$ is bounded on bounded sets by Lemma \ref{l:bbs}.
\end{proof}

For the benefit of the reader, we recall here a result we need next. 

\begin{lema}\label{split}\cite[Lemmas 3.3 and  3.4]{Vi} Let $K$ be a compact Hausdorff space. 
Let $\{G_i: i\in I\}$ be a family of open subsets of $K$. Let $G=\cup_{i\in I} G_i$. Then, for every $i\in I$ there exists a function $\varphi_i: G \longrightarrow [0,1]$ continuous in $G$ satisfying  $\varphi_i\prec G_i$ and such that $\bigvee_{i\in I} \varphi_i=\uno$ in $G$. Moreover,  let  $f\in C(K)^+$ satisfy $f\prec G$. Then, for every $i\in I$, the function  $f_i=\varphi_if$ belongs to $ C(K)^+$. Also,  $f_i\prec G_i$ and  $\bigvee_{i\in I} f_i=f$. In particular, for every $i\in I$,  $0\leq f_i\leq f$.
\end{lema}

Finally we can prove our decomposition result. It reduces the study of valuations to the positive case. 

\begin{teo}\label{t:jordan}
Let $V:E_+\longrightarrow \mathbb R$ be a continuous valuation. Then, there exist two continuous valuations $V^+, V^-:E_+\longrightarrow \mathbb R_+$ such that
$$
V=V^+-V^-.
$$
\end{teo}

\begin{proof}
Let $V:E_+\longrightarrow \mathbb R$ be as in the hypothesis. For every $f\in E_+$, we define
$$
V^+(f)=\sup\{V(g): \, 0\leq g \leq f\}.
$$

Assume for the moment that $V^+$ is a continuous valuation. In that case, the result follows easily:

First we note that it follows from $V(0)=0$ that $V^+(0)=0$ and that, for every $f\in E_+$, one has $V^+(f)\geq 0$. We define next $V^-=V^+-V$. Clearly, $V^-$ is also a continuous valuation and $V^-(0)=0$. By the definition of $V^+$, it follows that, for every $f\in E_+$, one has $V(f)\leq  V^+(f)$. Thus, $V^-(f)\geq 0$. And clearly we have $$V=V^+-V^-.$$

\smallskip

Therefore, we will finish if we show that
$V^+$ is  a continuous valuation.

First, we see that  it is a valuation. Let $f_1, f_2\in E_+$. We have to check that
\begin{equation}\label{igualdad}
V^+(f_1\vee f_2)+ V^+(f_1\wedge f_2) = V^+(f_1)+ V^+(f_2).
\end{equation}

Fix $\epsilon>0$. We choose $0\leq g_1\leq f_1$ such that $V^+(f_1)\leq V(g_1)+\epsilon$, and $0\leq g_2\leq f_2$ such that $V^+(f_2)\leq V(g_2)+\epsilon$.

Then,
\begin{align*}
V^+(f_1)+ V^+(f_2)&\leq V(g_1)+ V(g_2)+ 2\epsilon = V(g_1\vee g_2) + V(g_1\wedge g_2) + 2\epsilon\\
&\leq V^+(f_1\vee f_2) + V^+(f_1\wedge f_2) + 2\epsilon,
\end{align*}
where the last inequality follows from the fact that $0\leq g_1\vee g_2\leq f_1\vee f_2$ and $0\leq g_1\wedge  g_2\leq f_1\wedge f_2$. Since $\epsilon>0$ was arbitrary, this proves one of the inequalities in \eqref{igualdad}.

For the other one, fix again $\epsilon>0$. We choose $0\leq g \leq f_1\vee f_2 $ such that $V^+(f_1\vee f_2)\leq V(g)+\epsilon$, and $0\leq h \leq f_1\wedge f_2 $ such that $V^+(f_1\wedge f_2)\leq V(h)+\epsilon$.

Let us consider the compact Hausdorff space $K$ and the injective lattice homomorphism $j:C(K)\rightarrow E$ which maps the unit ball onto $[-f_1\vee f_2,f_1\vee f_2]$ given by Lemma \ref{l:local}. Note $\tilde V=V\circ j$ defines a continuous valuation on $C(K)$, and clearly $\widetilde{(V^+)}=(\tilde{V})^+$. For $f\in E$ with $|f|\leq f_1\vee f_2$, let us denote $\tilde f\in C(K)$ the unique function with $j(\tilde f)=f$. In particular, we have that $j(f_1\vee f_2)=\uno_K$.

Consider the sets
$$
A=\{ t\in K : \tilde{f_1}(t)\geq \tilde{f_2}(t)\}
$$
and
$$
B=\{ t\in K : \tilde{f_1}(t)< \tilde{f_2}(t)\}.
$$

According to Lemma \ref{rims}, there exists $\omega_1\in\mathcal U$ such that, for every $f\prec R(A,\omega_1)$ with $\|f\|_\infty\leq\lambda$ we have $|\tilde V(f)|\leq \epsilon$.

Since $\tilde{V}$ is continuous at $\tilde g$, there exists $\delta>0$ such that, $|\tilde{V}(\tilde g)-\tilde{V}(g')|<\epsilon$ for every $g'$ such that $\|\tilde g-g'\|_\infty<  \delta$. We define $g'=(\tilde g-\frac{\delta}{2})\vee 0$. Then, for every $t\in A$, it follows that
$$
g'(t)=\max\big\{\tilde g(t)-\frac{\delta}{2},0\big\}\leq \tilde g(t)\leq (\tilde f_1\vee \tilde f_2)(t)=\tilde f_1(t).
$$
Now, we can apply the uniform continuity of $\tilde g$ and $\tilde f_1$  to find $\omega_2\in\mathcal U$ such that $\omega_2\subset \omega_1$ and for every $t,s\in K$, if $(s,t)\in \omega_2$, then $|\tilde f_1(t)-\tilde f_1(s)|<\delta/4$ and  $|\tilde g(t)-\tilde g(s)|<\delta/4$. In particular, this implies that for every $t\in A_{\omega_2}$, $g'(t)\leq \tilde f_1(t)$. On the other hand, it is clear that $g'(t)\leq \tilde f_2(t)$ for $t\in B$.

Note that $K=A_{\omega_2} \cup B$ and $A_{\omega_2}\cap B=R(A,\omega_2)$.

We consider the functions  $\varphi_1\prec A_{\omega_2}$, $\varphi_2\prec B$  associated to the decomposition $K=A_{\omega_2}\cup B$ by Lemma \ref{split}. Then $\varphi_1\vee \varphi_2=\uno_K$. Let us define $g'_1=\varphi_1 g'$, $g'_2=\varphi_2 g'$, $h_1=\varphi_1 \tilde h$, $h_2=\varphi_2 \tilde h$ as in Lemma \ref{split}.

A simple verification yields \begin{itemize}
\item $g'=g'_1\vee g'_2$, $\tilde h=h_1\vee h_2$,

\item $g'_1\wedge g'_2\prec R(A,\omega_2)$, $h_1\wedge h_2\prec R(A,\omega_2)$,

\item $g'_1\wedge h_2\prec R(A,\omega_2)$, $h_1\wedge g'_2\prec R(A,\omega_2)$,

\item $0\leq g'_1\vee h_2 \leq \tilde f_1$,

\item $0\leq g'_2\vee h_1\leq  \tilde f_2$.

\end{itemize}

Therefore, we get
\begin{align*}
\tilde{V}^+(f_1\vee f_2)&+ \tilde{V}^+(f_1\wedge f_2) \leq  \tilde{V}(\tilde g)+ \tilde{V}(\tilde h) + 2\epsilon \leq \tilde{V}(g')+ \tilde{V}(\tilde h) + 3\epsilon \\
&=\tilde{V}(g'_1) + \tilde{V}(g'_2) - \tilde{V}(g'_1\wedge g'_2) + \tilde{V}(h_1) + \tilde{V}(h_2) - \tilde{V}(h_1\wedge h_2)+3\epsilon \\
&\leq\tilde{V}(g'_1) + \tilde{V}(h_2)+ \tilde{V}(g'_2) +\tilde{V}(h_1)+5\epsilon\\
&=\tilde{V}(g'_1\vee h_2)+ \tilde{V}(g'_1\wedge h_2)+ \tilde{V}(g'_2\vee h_1)+\tilde{V}(g'_2\wedge  h_1)+ 5\epsilon\\
&\leq  \tilde{V}(g'_1\vee h_2)+  \tilde{V}(g'_2\vee h_1) + 7\epsilon\leq \tilde{V^+}(\tilde f_1)+ \tilde{V}^+(\tilde f_2)+7\epsilon.
\end{align*}
Again, since $\epsilon>0$ was arbitrary and $\tilde V^+(\tilde f)=V^+(f)$, this finishes the proof of \eqref{igualdad}.

\smallskip

Let us see now that $V^+$ is continuous. Let us consider $f_0\in E^+$ and take $\epsilon>0$. There exists $g_0\in [0, f_0]$ such that $V^+(f_0)\leq V(g_0)+\epsilon$.

Since $V$ is continuous at $f_0$ and $g_0$,  there exists $\delta>0$ such that for every $f, g \in E^+$ with $\|f_0-f\|_E<\delta$ and $\|g_0-g\|_E<\delta$, we have  $|V(f_0)-V(f)|<\epsilon$ and  $|V(g_0)-V(g)|<\epsilon$.

Let now $f\in E^+$ be such that $\|f_0-f\|_E<\delta$. Pick $g\in E^+$ with $0\leq g \leq f$ such that $V^+(f)\leq V(g) + \epsilon$.

Note that  $\|g_0\wedge f - g_0\|_E<\delta$ and $\|g\vee f_0 - f_0\|_E<\delta$. Then, we have
$$
V^+(f)\geq V(g_0\wedge f) \geq V(g_0)-\epsilon\geq V^+(f_0)-2\epsilon,
$$
and
\begin{eqnarray*}
V^+(f)&\leq& V(g) + \epsilon =V(g\wedge f_0) + V(g\vee f_0) -  V(f_0) +\epsilon \\
&\leq& V(g\wedge f_0) + |V(g\vee f_0) -  V(f_0) |+ \epsilon \leq V^+(f_0) + 2\epsilon.
\end{eqnarray*}

Hence, $$|V^+(f_0)-V^+(f)|<2\epsilon$$ and $V^+$ is continuous as claimed.
\end{proof}

\section{Boundedness on norm-bounded sets}\label{s:bnbs}

In Section \ref{s:bbs} we proved that continuous valuations on Banach lattices are bounded on order bounded sets. Often, one would need a stronger statement, namely that such valuations are bounded on {\em norm}-bounded sets. We will see with an example that this result is, in general not true. Nevertheless, we will show that the result {\em is} true for a large class of Banach lattices. 

Recall that a Banach lattice is said to satisfy a lower $q$-estimate for some $q<\infty$ if there exists $M>0$ such that for every choice of pairwise disjoint elements $(x_i)_{i=1}^n$ in $E$ we have
$$
\Big(\sum_{i=1}^n\|x_i\|^q\Big)^{\frac1q}\leq M \Big\|\sum_{i=1}^n x_i\Big\|.
$$

For example, the space $L_p(\mu)$ satisfies a lower $q$-estimate for every $q\geq p$. In general, the constant $M$ above can be taken to be 1, up to an equivalent renorming. Also, recall that every Banach lattice with finite cotype satisfies a lower $q$-estimate for some $q<\infty$, and that these spaces are always order continuous. We refer to \cite[1.f]{LT2} for a detailed discussion of Banach lattices with this property.

\begin{teo}\label{normboundedsets}
Let $E$ be a Banach lattice of measurable functions over a purely non-atomic $\sigma$-finite measure space $(\Omega,\Sigma,\mu)$. If $E$ satisfies a lower $q$-estimate for some $q<\infty$, then every valuation on $E$ which is continuous at $0$ is bounded on norm bounded sets.
\end{teo}

Before the proof, we need to recall some terminology and an auxiliary lemma.  Given a measure space $(\Omega,\Sigma,\mu)$, we will consider the Fr\'echet-Nikodym metric space $X_{(\Omega,\Sigma,\mu)}$ associated to it: this consists of classes of sets in $\Sigma$ (where two sets are identified if they defer by a set of $\mu$-measure zero), equipped with the metric $d(A,B)=\mu(A\Delta B)$. This can be identified with the subset of $L_1(\mu)$ consisting of all characteristic functions $\chi_A$ for $A\in \Sigma$, with the metric induced by the restriction of the $L_1$-norm (cf. \cite[1.12(iii)]{Bo}).

The following auxiliary result is probably known to measure theorists, but we have not found a reference for it. 

\begin{lema}\label{noatoms}
Let $(\Omega,\Sigma,\mu)$ be a finite measure space without atoms. Then the Fr\'echet-Nikodym metric space $X_{(\Omega,\Sigma,\mu)}$ is connected.
\end{lema}

\begin{proof}
For simplicity, we will assume $L_1(\Omega,\Sigma,\mu)$ is separable. Suppose that $X_{(\Omega,\Sigma,\mu)}$ is not connected, and let $\mathcal U,\mathcal V$ be nonempty clopen sets such that $\mathcal U\cup \mathcal V=X_{(\Omega,\Sigma,\mu)}$ and $\mathcal U\cap\mathcal V=\emptyset$. Without loss of generality, we can assume $\Omega\in \mathcal U$.

We claim that there is a maximal element $B\in \mathcal V$ in the sense that $B\subset C$ with $\mu(C\backslash B)>0$ implies $C\notin \mathcal V$. In order to see this, we will make use of Zorn's Lemma. Let $(A_i)_{i\in I}$ be a chain in $\mathcal V$. Since $(\chi_{A_i})_{i\in I}\subset L_1(\mu)$ are order bounded by $\chi_\Omega\in L_1(\mu)$, and $L_1(\mu)$ is Dedekind complete, it follows that there is $A\in \Sigma$ such that $\sup_{i\in I}\chi_{A_i}=\chi_A$ in $L_1(\mu)$. We can thus extract a subsequence $(i_k)_{k\in\mathbb N}\subset I$ such that $\chi_{A_{i_k}}\leq\chi_{A_{i_{k+1}}}$ and $\chi_A=\sup_{k\in \mathbb N}\chi_{A_{i_k}}$, so by the monotone convergence theorem it follows that
$$
\mu(A\Delta A_{i_k})=\|\chi_A-\chi_{A_{i_k}}\|_1\underset{k\rightarrow\infty}{\longrightarrow}0.
$$
Since $(A_{i_k})_{k\in\mathbb N}\subset \mathcal V$ and $\mathcal V$ is closed, it follows that $A\in \mathcal V$ as claimed. Therefore, Zorn's Lemma guarantees the existence of a maximal element $B\in\mathcal V$.

Note that $\mu(B)<\mu(\Omega)$. Otherwise, $B\in\mathcal U\cap \mathcal V$ which is impossible. Moreover, since $\mathcal V$ is open, there is $\delta>0$ such that if $C\in\Sigma$ satisfies $\mu(B\Delta C)<\delta$, then $C\in\mathcal V$. Now, since $\mu$ has no atoms, we can find $C\subset \Omega\backslash B$ such that $\mu(C)<\delta$. It follows that
$$
\mu(B\Delta (B\cup C))=\mu(C)<\delta,
$$
so $B\cup C\in \mathcal V$. This is a contradiction with the maximality of $B$, so the Lemma is proved.
\end{proof}

\begin{proof}[Proof of Theorem \ref{normboundedsets}]
Without loss of generality, taking an equivalent norm, we can assume that
$$
\Big\|\sum_{i=1}^m f_i\Big\|^q\geq\sum_{i=1}^m\|f_i\|^q,
$$
whenever $(f_i)_{i=1}^m\subset E$ are pairwise disjoint.

Since $V:E\rightarrow\mathbb R$ is continuous at $0$, there is $\delta>0$ such that $|V(f)|\leq 1$ whenever $\|f\|\leq \delta$.

We claim that for every $f\in E$, it holds that
$$
|V(f)|\leq\Big(\frac{\|f\|}{\delta}\Big)^q+2.
$$

Indeed, let $f\in E$. If $\|f\|\leq \delta$, then the claim holds trivially, so assume this is not the case. By the order continuity of $E$, there is $A_0\subset \Omega$ such that $\mu(\Omega\backslash A_0)<\infty$ and $\|f\chi_{A_0}\|\leq\delta$. Since $E$ satisfies a lower q-estimate we have that 
$$
\|f\chi_{\Omega\backslash A_0}\|^q\leq\|f\|^q-\delta^q.
$$
Let us consider the following function defined on the Fr\'echet-Nikodym space:
$$
\begin{array}{cccc}
  \Phi_f:&X_{(\Omega\backslash A_0,\Sigma,\mu)}&\longrightarrow &\mathbb R_+ \\
   & A &\longmapsto &\|f\chi_A\|
\end{array}
$$
Note that, as $E$ is order continuous, in particular $\Phi_f$ is continuous with respect to the metric $d(A,B)=\mu(A\Delta B)$. Moreover, by Lemma \ref{noatoms}, $X_{(\Omega\backslash A_0,\Sigma,\mu)}$ is connected, so the set $\Phi_f(X_{(\Omega\backslash A_0,\Sigma,\mu)})$ is connected in $\mathbb R_+$.

Now, if $\|f\chi_{\Omega\backslash A_0}\|\leq \delta$ the claim holds trivially. Suppose now that $\|f\chi_{\Omega\backslash A_0}\|>\delta$. Since $\Phi_f(\emptyset)=0$ and $\Phi_f(\Omega\backslash A_0)=\|f\chi_{\Omega\backslash A_0}\|>\delta$, by connectedness of  $\Phi_f(X_{(\Omega,\Sigma,\mu)})$, there exist $A_1\subset \Omega\backslash A_0$, such that
$$
\|f\chi_{A_1}\|=\delta.
$$

Using the lower q-estimate, it follows that
$$
\|f\chi_{\Omega\backslash (A_0\cup A_1)}\|^q\leq\|f\|^q-2\delta^q.
$$
Repeating the process with $f\chi_{\Omega\backslash A_1}$, inductively we obtain a finite family of pairwise disjoint sets $A_0,A_1,\ldots A_n$ such that
$$
\|f\chi_{A_i}\|=\delta
$$
for $i=0,\ldots, n$ and $\|f\chi_{\Omega\backslash\bigcup_{i=0}^n A_i}\|^q\leq \|f\|^q-(n+1)\delta^q$, until $\|f\|^q-(n+1)\delta^q\leq \delta^q$. In other words, if we take $n=[\frac{\|f\|^q}{\delta^q}]$, then
$$
V(f)=\sum_{i=0}^n V(f\chi_{A_i})+V(f\chi_{\Omega\backslash\bigcup_{i=0}^n A_i})\leq n+2\leq \Big(\frac{\|f\|}{\delta}\Big)^q+2,
$$
as claimed.
\end{proof}

As we mentioned, Theorem \ref{normboundedsets} applies to spaces $L_p(\mu)$ for $1\leq p<\infty$. More generally, it is well-known that the Orlicz space $L_M(0,1)$ satisfies a lower $q$-estimate for some $q<\infty$ whenever the function $M$ satisfies the $\Delta_2$-conditioin at $\infty$: $\underset{t\rightarrow\infty}{\lim\sup}\frac{M(2t)}{M(t)}<\infty$ (cf. \cite{LT2}).

On the other hand a continuous valuation on an atomic space need not be bounded on norm bounded sets.

\begin{prop}\label{c_0unbounded}
Consider $V:c_0^+\rightarrow\mathbb R^+$ which for $x=(x_n)_{n\in\mathbb N}\in c_0^+$ is given by
$$
V(x)=\sum_{n\in\mathbb N} n x_n^n.
$$
$V$ is a continuous valuation and $V(e_n)=n$ for each $n\in\mathbb N$, where $(e_n)$ is the unit vector basis of $c_0$.
\end{prop}

\begin{proof}
Recall that for every $\varepsilon\in(0,1)$ we have
$$
\sum_{n\in\mathbb N} n \varepsilon^n=1+\frac{(2-\varepsilon)\varepsilon}{(1-\varepsilon)^2}.
$$
In particular, given $x=(x_n)_{n\in\mathbb N}\in c_0^+$, take $N\in\mathbb N$ such that for $n\geq N$, $0\leq x_n<\frac12$. Thus,
$$
\sum_{n\in\mathbb N} n x_n^n=\sum_{n< N} n x_n^n+\sum_{n\geq N} n x_n^n\leq \sum_{n< N} n x_n^n+\sum_{n\geq N} n \frac{1}{2^n}<\infty,
$$
so $V$ is well defined. Moreover, given $x=(x_n)_{n\in\mathbb N}, \,y=(y_n)_{n\in\mathbb N}\in c_0^+$ we have
\begin{eqnarray*}
V(x\vee y)+ V(x\wedge y)&=&\sum_{n\in\mathbb N} n\max\{x_n,y_n\}^n+\sum_{n\in\mathbb N} n\min\{x_n,y_n\}^n\\
&=&\sum_{n\in\mathbb N} n x_n^n+\sum_{n\in\mathbb N} n y_n^n=V(x)+V(y).
\end{eqnarray*}

For continuity, let $x=(x_n)_{n\in\mathbb N}\in c_0^+$ and $\varepsilon\in(0,1)$. Take $N\in\mathbb N$ such that $0\leq x_n\leq \varepsilon/2$ for $n> N$. Using the continuity of the real function $a\mapsto a^n$, we can take $\delta\in(0,\varepsilon/2)$ such that for each $n\leq N$ if $a,b\in[0,\|x\|+1]$ and $|a-b|<\delta$ then $|a^n-b^n|\leq \frac{n\varepsilon}{2^n}$. Hence, if $y\in c_0^+$ is such that $\|x-y\|_\infty<\delta$, then
$$
|V(x)-V(y)|=\sum_{n\in\mathbb N} n|x_n^n-y_n^n|\leq\sum_{n\leq N} \frac{\varepsilon}{2^n}+\sum_{n> N}nx_n^n +\sum_{n> N}ny_n^n\leq\varepsilon+2\frac{(2-\varepsilon)\varepsilon}{(1-\varepsilon)^2}.
$$
\end{proof}

This can be easily extended to sequence spaces like $\ell_p$, Orlicz spaces $\ell_\varphi$, and actually, to any atomic Banach lattice, in the following sense:

\begin{prop}
Let $E$ be an atomic Banach lattice with the order given by an unconditional basis. There exist continuous valuations on $E$ which are not bounded on norm-bounded sets.
\end{prop}

\begin{proof}
Let $(u_n)\subset E$ be the unconditional (normalized) basis giving the order in $E$. Let us define $V:E\rightarrow \mathbb R_+$ by
$$
V\Big(\sum_n a_n u_n\Big)=\sum_n n |a_n|^n.
$$
By the unconditionality of $(u_n)$, there is $C>0$ such that $\|\sum_n a_n u_n\|\geq C \sup_n |a_n|$ for every scalars $(a_n)$. Thus, by Proposition \ref{c_0unbounded}, $V$ defines a continuous valuation on $E$. Since $\|u_n\|=1$ and $V(u_n)=n$, in particular $V$ is not bounded on norm-bounded sets.
\end{proof}

\section{Valuations vs. orthogonally additive functionals}\label{s:oa}

A continuous mapping $\Phi:E\rightarrow \mathbb R$ is orthogonally additive if $\Phi(x+y)=\Phi(x)+\Phi(y)$ whenever $|x|\wedge|y|=0$. Clearly, every valuation is orthogonally additive.

Since orthogonal additivity is much easier to check that the condition of being a valuation, it would be very convenient to know when both notions can coincide. On the one hand, we will see that, in general, orthogonally additive mappings need not be valuations (Proposition \ref{p:min}). On the other hand, we see next that orthogonally additive mappings coincide with valuations for  $\sigma$-Dedekind complete Banach lattices: Recall that $E$ is $\sigma$-Dedekind complete when every order bounded sequence $(x_n)\subset E$ has a supremum. The class of $\sigma$-Dedekind complete Banach lattices includes, among others, order continuous Banach lattices, dual Banach lattices, the space $B(K)$ of bounded Borel functions on a compact Hausdorff space $K$, and spaces $C(K)$ of continuous functions on a basically disconnected compact Hausdorff space $K$ (cf. \cite{LT2}).

If $E$ is $\sigma$-Dedekind complete, to each $x\in E_+$ we can associate a (band) projection $P_x$ given by
$$
P_x(y)=\bigvee_{n=1}^\infty (y\wedge n x),
$$
for $y\in E_+$, and by $P_x(y)=P_x(y_+)-P_x(y_-)$ for arbitrary $y\in E$. Note that for $x,y\in E_+$ it follows that
\begin{equation}\label{eq:bands}
x\wedge (y-P_x(y))=0.
\end{equation}

\begin{prop}\label{p:dedekind}
Let $E$ be a $\sigma$-Dedekind complete Banach lattice. Every orthogonally additive functional $\Phi:E\rightarrow \mathbb R$ is a valuation.
\end{prop}

\begin{proof}
Given $f,g\in E$, let us consider the band projections given by
$$
P_1=P_{(f-g)_+}, \quad P_2=I-P_1.
$$
Note that we have the following identities:
\begin{enumerate}
\item[(i)] $P_1(f\vee g)=P_1(f)$,
\item[(ii)] $P_2(f\vee g)=P_2(g)$,
\item[(iii)] $P_1(f\wedge g)=P_1(g)$,
\item[(iv)] $P_2(f\wedge g)=P_2(f)$.
\end{enumerate}
Indeed, let us check (i)
$$
P_1(f\vee g)-P_1(f)=P_1(0\vee(g-f))=P_1((f-g)_-)=0.
$$
The other identities are similar.

Since the ranges of $P_1$ and $P_2$ are mutually disjoint, and $\Phi$ is orthogonally additive, we get
$$
\Phi(f\vee g)=\Phi(P_1(f\vee g))+\Phi(P_2(f\vee g))=\Phi(P_1f)+\Phi(P_2g).
$$
Similarly, we have
$$
\Phi(f\wedge g)=\Phi(P_1(f\wedge g))+\Phi(P_2(f\wedge g))=\Phi(P_1g)+\Phi(P_2f).
$$

Thus, we have
$$
\Phi(f\vee g)+\Phi(f\wedge g)=\Phi(P_1f)+\Phi(P_2g)+\Phi(P_1g)+\Phi(P_2f)=\Phi(f)+\Phi(g).
$$
\end{proof}

Note that no continuity assumption is necessary in the previous proposition. In particular, a standard density argument allows us to construct examples of non $\sigma$-Dedekind complete lattices for which every orthogonally additive continuous mapping is a valuation.

However, in general not every orthogonally additive functional is a valuation, as the following shows:

\begin{prop}\label{p:min}
The mapping $\phi:C(K)\rightarrow \mathbb R$ given by $$\phi(f)=\min\{|f(t)|:t\in K\}$$ is continuous and satisfies $\phi(0)=0$. Moreover,
\begin{enumerate}

\item $\phi$ is orthogonally additive if  $K$ is connected (and only in this case). 

\item $\phi$ is not a valuation.
\end{enumerate}
\end{prop}

\begin{proof}
It is clear that $\phi$ is continuous and $\phi(0)=0$. First, let us see that $\phi$ is not a valuation. Indeed, consider a partition of $K$ into two sets $A,B$ with $A\cap B=\emptyset$ and functions $f_A,g_B\in C(K)_+$ such that $f_A(t)=1$ for every $t\in A$, $g_B(t)=1$ for every $t\in B$ and for some $t_A\in A$ and $t_B\in B$ we have $f_A(t_B)=0=g_B(t_A)$. It follows that
$$
\phi(f_A)=\phi(g_B)=\phi(f_A\wedge g_B)=0,
$$
while $\phi(f_A\vee g_B)=1$. Therefore, $\phi$ cannot be a valuation.

Suppose that  $K$ is connected, let $f,g\in C(K)$ such that $f\perp g$ and set
$$
A=\{t\in K:f(t)\neq0\}, \quad B=\{t\in K:g(t)\neq0\}.
$$
Clearly, $A\cap B=\emptyset$.

Suppose first that $\phi(f\vee g)\neq0$, then $A\cup B=K$. Since $K$ is connected, in this case it follows that either $A=\emptyset$ or $B=\emptyset$. If $A=\emptyset$, then $f(t)=0$ for every $t\in K$, which yields
$$
\phi(f+g)=\phi(g)=\phi(f)+\phi(g).
$$
Similarly, if $B=\emptyset$, then we have $\phi(f+g)=\phi(f)=\phi(f)+\phi(g).$

Finally, suppose now that $\phi(f\vee g)=0$. In this case, it follows that $\phi(f)=\phi(g)=0$, so we also have $\phi(f+g)=\phi(f)+\phi(g).$

Suppose now that $K$ is not connected. Let $A,B\subset K$ be pairwise disjoint non-empty clopen sets such that $A\cup B=K$. We have that
$$
\phi(\chi_A+\chi_B)=\phi(\chi_K)=1,
$$
while $\phi(\chi_A)=\phi(\chi_B)=0$, and since $\chi_A\perp \chi_B$, it follows that $\phi$ is not orthogonally additive.
\end{proof}

\section{Integral representation of valuations}\label{s:DO}

The coincidence just proved of orthogonally additive mappings and valuations on a large class of Banach lattices allows us to profit from several results and techniques developed in the 1960's for the study of orthogonally additive mappings. After a series of works, in \cite{DO68, DO69} the authors prove an integral representation theorem valid for orthogonally additive mappings defined on a quite comprehensive class of function spaces over a measure space. 

This section starts with an upgraded version of that result, Theorem \ref{t:DOsigmafinitas}: we state it in the terminology of valuations and extend it to Banach lattices of measurable functions on spaces of $\sigma$-finite measure. The original result was stated for the case of finite measure, which applies, for instance, to $L_p(S^{n-1})$ but not directly to $L_p(\mathbb R^n)$. 

In the rest of the section we extract different consequences of Theorem \ref{t:DOsigmafinitas}. First we show how that result, together with the main result of \cite{TraVi, TraVi2}, yields a quite general factorization property of valuations through valuations on $L_1$: Theorem \ref{t:L1factorization}. Next, we show the simpler and more convenient form that Theorem \ref{t:DOsigmafinitas} adquires when the valuation has enough invariance with respect to measure preserving transformation. This result includes and extends the main results of \cite{Ts} and \cite{Ko}. Finally, we study with finer detail properties of Theorem \ref{t:DOsigmafinitas} in the special case of $L_p$ spaces, due to their special significance. 

Throughouth, $(\Omega, \Sigma, \mu)$ will be a $\sigma$-finite measure space. We denote by $L_0(\mu)$ the space of (equivalence classes of) measurable functions defined on $\Omega$, which is a Hausdorff topological vector lattice when equipped with the  topology of convergence in measure. $X$ will be an order continuous Banach lattice of functions on $(\Omega, \Sigma, \mu)$ for which the formal inclusion $X\hookrightarrow L_0(\mu)$ is continuous and its image is an ideal in $L_0(\mu)$ containing the characteristic functions of sets with finite measure. For brevity, we refer to this saying that $X$ is represented as a function space on $(\Omega, \Sigma, \mu)$ (compare with the notion of K\"othe function space given in \cite[Definition 1.b.17]{LT2}).

We will also need the following notation: Given a measure space $(\Omega, \Sigma, \mu)$ as before, we say that a function $$K:\mathbb R\times \Omega\longrightarrow \mathbb R$$ is a strong Carath\'eodory function if $K(\lambda, \cdot) $ is measurable for every $\lambda \in \mathbb R$ and $K(\cdot, t)$ is continuous for $\mu$-almost every $t\in \Omega$.

Recall that it follows from Proposition \ref{p:dedekind} that, in one such $X$, a function $V:X\longrightarrow \mathbb R$ is a valuation if and only if $V$ is orthogonally additive.

With these observations, restricting to order continuous Banach lattices, we state the main result of \cite{DO68, DO69}, extending it to the case of $\sigma$-finite measures.

\begin{teo}\cite[Theorem 3.2]{DO69}\label{t:DOsigmafinitas}
Let $X$ be an order continuous Banach lattice represented as a function space on $(\Omega, \Sigma, \mu)$, and let $V:X\longrightarrow \mathbb R$ be a continuous valuation. Then,  there exists a strong Carath\'eodory function  $K(\lambda, t):\mathbb R \times \Omega \longrightarrow \mathbb R$ such that, for every $f\in X$, $$V(f)=\int_\Omega K(f(t),t)d\mu(t).$$

Conversely, if $K(\lambda, t):\mathbb R \times \Omega \longrightarrow \mathbb R$ is a strong Carath\'eodory function such that, for every $f\in X$, $\int_\Omega K(f(t),t)d\mu(t)<\infty$, then the function $V:X\longrightarrow \mathbb R$ defined by  $$V(f)=\int_\Omega K(f(t),t)d\mu(t)$$ is a continuous valuation. 
\end{teo}

\begin{proof}
We prove first the first part of the statement. 
Let $V$ be as above.  First we note that, by Theorem \ref{t:jordan}, we may assume without loss of generality that $V$ takes values in $\mathbb R_+$. This fact will be used without further mention in several points along the proof. 

If $\mu$ is finite, the result is proven in \cite[Theorem 3.2]{DO69}. Suppose now that $\mu$ is $\sigma$-finite. Then, there exists a sequence $(\Omega_n)_{n\in \mathbb N}$ of subsets of $\Omega$ such that $\Omega=\bigcup_{n\in \mathbb N} \Omega_n$ and $\mu(\Omega_n)<\infty$ for every $n\in \mathbb N$. We may and do assume that the sequence $(\Omega_n)_{n\in \mathbb N}$ is increasing. 

For every $n\in \mathbb N$ we consider $$X_n:=\{ f\chi_{\Omega_n}: \, f\in X\}$$

It is easy to see that $X_n$ is an order continuous  Banach lattice corresponding to $X$ in the finite measure space $(\Omega_n, \Sigma_n, \mu)$, where $\Sigma_n:=\{A\cap\Omega_n:\, A\in \Sigma\}$.  It is also clear that $X_n$ is a closed subspace of $X$ (actually, a band).

On $X_n$ we consider the application $V_n=V_{|_{X_n}}$. Clearly $V_n$ is a continuous valuation. Therefore, we can apply \cite[Theorem 3.2]{DO69} and we obtain a representing strong Carath\'eodory function $K_n:\mathbb R\times \Omega_n\longrightarrow \mathbb R.$

That is, for every $f\in X_n$, $$V_n(f)=\int_{\Omega_n} K_n(f(t), t) d\mu(t).$$

If we extend $K_n$ to all of $\Omega$ by defining $K_n(\lambda, t)=0$ whenever $t\in \Omega\setminus \Omega_n$, it is clear  that, for every $n<m$, $K_n(\lambda, t)=K_m(\lambda, t\chi_{\Omega_n)}(t))$. Therefore, for every $\lambda\in \mathbb R$, $K_n(\lambda, \cdot)$ is an increasing monotone sequence. Moreover, given $t\in \Omega$, there exists $n_t\in \mathbb N$ such that $t\in \Omega_{n_t}$, and, hence, $\sup_n K_n(\lambda, t)=K_{n_t}(\lambda, t)$.

So, for every fixed $\lambda \in \mathbb R$ we can define $K(\lambda, t)=\sup_n K_n(\lambda, t)$. Thus defined, $K(\lambda, \cdot)$ is measurable for every $\lambda\in \mathbb R$. This is one of the conditions required for $K$ to be strong Carath\'eodory.

The continuity in the first variable, also follows: We have seen that, for every $t\in \Omega_n$ and for every $\lambda\in \mathbb R$,  $K(\lambda, t)=K_n(\lambda, t)$.  Note  that, for every $n\in \mathbb N$, $K_n$ is a strong Carath\'eodory function. Therefore, it is continuous in the first variable outside of  a set $A_n$ of zero measure. We consider the set $A:=\bigcup_n A_n$. Then, $\mu(A)=0$. We consider now $t\not \in A$. There exists $n\in \mathbb N$ such that $t\in \Omega_n$. Therefore, for every $\lambda\in \mathbb R$, $K(\lambda, t)=K_n(\lambda, t)$. Since $K_n(\cdot, t)$ is continuous, we obtain that $K(\cdot, t)$ is also continuous.

We see next that $K$ provides the proper integral representation of $V$. 

Let $f\in X$. Since $X$ is order continuous and $\Omega_n\nearrow \Omega$, we get that the sequence $(f\chi_{\Omega_n})_{n\in \mathbb N}$ converges in norm to $f$. Therefore, the continuity of $V$ implies that $V(f\chi_{\Omega_n})\rightarrow V(f)$ as $n$ grows to infinity. 

At the same time, note that $f\chi_{\Omega_n}\in X_n$. Therefore $$V(f\chi_{\Omega_n})=V_n(f\chi_{\Omega_n})=\int_{\Omega_n} K_n(f(t)\chi_{\Omega_n}(t), t) d\mu(t)= \int_\Omega K_n(f(t), t) d\mu(t). $$

Now, the Monotone Convegence Theorem implies that  $\int_\Omega K_n(f(t), t) d\mu(t)$ converges to $\int_\Omega K(f(t), t) d\mu(t)$ as $n$ grows to infinity. Therefore, 

$$V(f)=\int_\Omega K(f(t), t) d\mu(t).$$
This finishes the proof of the first part of the statement. 

We suppose now that $K(\lambda, t):\mathbb R \times \Omega \longrightarrow \mathbb R$ is a strong Carath\'eodory function such that, for every $f\in X$, $$\int_\Omega K(f(t),t)d\mu(t)<\infty.$$

Then, we can define $V$ as in the statement. Clearly $V$ is orthogonally additive and, therefore, a valuation by virtue of Proposition \ref{p:dedekind}. We have to check the continuity of $V$. First we consider the case when $\mu$ is finite. The proof is very shortly sketched in \cite{DO68, DO69}. For the sake of clarity we write it next in detail:

Associated to $K$ we will consider an application $F:X\longrightarrow L_1(\mu)$ defined by $$F(f)(t)=K(f(t), t).$$
We check that $F$ is well defined: 
First, it is easy to see that, for every $f\in X$, the function $K(f(\cdot), \cdot):\Omega\longrightarrow \mathbb R$ is measurable (for a proof, see \cite[p. 349]{KZPS}). Since we also have that $\int_\Omega K(f(t),t)d\mu(t)<\infty$, we get that indeed $F(f)\in L_1(\mu)$ for every $f\in X$. 

Let us see that $F$ is continuous:
 
Let $f\in X$ and let $(f_n)_{n\in \mathbb N}$ be a sequence convergent to $f$ in norm. By the continuity of the inclusion $X\hookrightarrow L_0(\mu)$, it follows that $f_n$ also converges to $f$ in measure. We consider the sequence $(F(f_n))_{n\in \mathbb N}$. First, we see that it converges to $F(f)$ in measure: remember that, for a finite measure $\mu$, a sequence $(g_n)_n\in \mathbb N$ converges to $g$ in measure if and only if every subsequence $(g_{n_k})_{k\in \mathbb N}$ contains a subsequence that converges to $f$ $\mu$-almost everywhere. 

So, let $(F(f_{n_k}))_{k\in \mathbb N}$ be a subsequence of $(F(f_n))_{n\in \mathbb N}$. Since $(f_n)_{n\in \mathbb N}$ converges to $f$ in measure, there exists a subsequence, which we just call $(f_k)$ of $(f_{n_k})$ such that $(f_k)$ converges to $f$ $\mu$-a.e. Now, the fact that $K$ is strong Carath\'eodory implies that $(F(f_k))$ converges $\mu$-a.e. to $F(f)$. 

Recall now that that a set $\mathcal F\subset L_1(\Omega,\Sigma,\mu)$ is uniformly integrable if for every $\epsilon>0$ there exists $\delta>0$ such that for every $B\in \Sigma$ with $\mu(B)<\delta$
$$
\sup_{f\in\mathcal F}\int_B |f| d\mu<\epsilon.
$$

The following fact is folklore (for a proof see, for instance, \cite[Lemma 3.6]{TraVi2}): if  $(g_n)_{n\in\mathbb N}\subset L_1(\Omega,\Sigma,\mu)$ with $\mu(\Omega)<\infty$,  $g_n\rightarrow g$ $\mu$-almost everywhere, and the sequence $(g_n)_{n\in\mathbb N}$ is uniformly integrable, then $g_n\rightarrow g$ in  $L_1$ norm. 

So, we prove next that the set $\mathcal F:=\{F(f_n): n\in \mathbb N\}$ is uniformly integrable. We follow the ideas of \cite[Lemma 17.3]{KZPS}. Suppose $\mathcal F$ is not uniformly integrable. Then there exists $\epsilon>0$, a subsequence $F(f_{n_k})_{k\in \mathbb N}$ and a sequence of sets $(A_k)_{k\in \mathbb N}$ such that $\mu(A_k)\rightarrow 0$ and $$\int_{A_k} |K(f_{n_k}(t), t)|d\mu(t)>\epsilon.$$

Without loss of generality we may assume that $\sum_k \mu(A_k)<\infty.$ For every $k\in \mathbb N$ we define $$B_k:=\bigcup_{i=k}^\infty A_i. $$

Since $\mu(B_k)\rightarrow 0$, for every $k\in \mathbb N$ there exists a natural number $\eta(k)$ such that 

$$\int_{A_k\setminus B_{\eta(k)}} |K(f_{n_k}(t), t)|d\mu(t)>\frac{\epsilon}{2}.$$

For every $k\in \mathbb N$, we consider the sets $A'_k:=A_k \setminus B_{\eta(k)}$ and define the subsequence $k_1=1$, $k_2=\eta(k_1)$, \ldots, $k_n=\eta(k_{n-1})$\ldots. Then we have that the sets in the sequence $(A'_{k_j})_{j\in \mathbb N}$ are mutually disjoint, $\mu(A'_{k_j})\rightarrow 0$ as $j\rightarrow \infty$ and $\int_{A'_{k_j}} |K(f_{n_k}(t), t)|d\mu(t)>\frac{\epsilon}{2}.$

We define now the function 
$$
g=\sum_{j=1}^\infty f_{k_j}\chi_{A_{k_j}}.
$$
To see that $g\in X$, we can assume without loss of generality that $\|f-f_{k_j}\|\leq\frac1{2^j}$, therefore for every $N\in\mathbb N$ we have
\begin{eqnarray*}
\Big\|\sum_{j=N}^\infty f_{k_j}\chi_{A_{k_j}}\Big\|&\leq &\Big\|\sum_{j=N}^\infty f\chi_{A_{k_j}}\Big\|+\Big\|\sum_{j=N}^\infty (f-f_{k_j})\chi_{A_{k_j}}\Big\|\\
&\leq&\|f\chi_{\bigcup_{j=N}^\infty A_{k_j}}\|+\sum_{j=N}^\infty\|f-f_{k_j}\|.
\end{eqnarray*}
Since $X$ is order continuous it follows that $\lim_{N\rightarrow \infty}\|f\chi_{\bigcup_{j=N}^\infty A_{k_j}}\|=0$, so the above estimate shows that $g\in X$.

But now, for every $m\in \mathbb N$,
$$
F(g)\geq \sum_{j=1}^m F(f_{k_j}\chi_{A_{k_j}})\geq \frac{m \epsilon}{2},
$$
which is a contradiction. This finishes the proof for the finite measure case.

Finally, if $\mu$ is $\sigma$-finite, we can do a standard change of density argument to finish the proof: let $g:\Omega\rightarrow \mathbb R$ such that $g(t)>0$ $\mu$-almost everywhere and 
$$
\int_\Omega g(t)d\mu(t)=1.
$$ 
Consider the finite measure $\widetilde\mu$ on $(\Omega,\Sigma)$, given by 
$$
\widetilde\mu(A)=\int_A g(t)d\mu(t),
$$
and let $\widetilde X=\{f:\Omega\rightarrow\mathbb R / fg\in X\}$. It follows that $\widetilde X$ endowed with the norm $\|f\|_{\widetilde X}=\|fg\|_X$ is a Banach lattice represented as a function space on the finite measure space $(\Omega,\Sigma,\widetilde\mu)$. Hence, if $K(\lambda, t):\mathbb R \times \Omega \longrightarrow \mathbb R$ is a strong Carath\'eodory function such that, for every $f\in X$, $\int_\Omega K(f(t),t)d\mu(t)<\infty$, then we can consider 
$$
\widetilde K(\lambda,t)=\frac{K(\lambda g(t),t)}{g(t)},
$$
which is clearly a strong Carath\'eodory function. Since $f\in \widetilde X$, if and only if $fg\in X$, and for $f\in \widetilde X$ we have that
$$
\int_\Omega \widetilde K(f(t),t)d\widetilde \mu(t)= \int_\Omega K(f(t)g(t),t)d\mu(t)<\infty,
$$
it follows that the valuation $\widetilde V:\widetilde X\rightarrow \mathbb R$ given by $\widetilde V(f)=\int_\Omega \widetilde K(f(t),t)d\widetilde \mu(t)$ is continuous by the previous argument for finite measures. Therefore, if $f_n\rightarrow f$ in $X$, we have that $\frac{f_n}{g}\rightarrow \frac{f}{g}$ in $\widetilde X$, so 
$$
V(f_n)=\widetilde V\Big(\frac{f_n}{g}\Big)\longrightarrow V\Big(\frac{f}{g}\Big)=V(f),
$$
as claimed.
\end{proof}

The results of \cite{DO68, DO69} are only stated for the case of lattices represented on spaces of finite measure. But, in these cases, they  cover a context more general that the one we have chosen to present above. In particular they cover  the case of $L_p$ spaces, $0<p<1$. We refer the interested reader to \cite{DO68, DO69} for further detail.

\subsection{$L_1$ factorization of valuations in Banach lattices}

Universal factorization results are very useful in several areas of mathematics. Often  they reveal some hidden  structure underlying a given problem. In this section we show how some of the main results in \cite{DO68}, \cite{DO69}, \cite{Vi}, \cite{TraVi}, \cite{TraVi2} can be presented as a factorization result for continuous valuations on a quite comprehensive class of Banach lattices. 
In particular, it will cover  most of the ``reasonable'' $C(K)$, $L_p$ and Orlicz spaces.

Before we state it, we recall that \cite[Theorem 1.1]{TraVi} shows that a continuous valuation  $V:C(K)\longrightarrow \mathbb R$ can be extended to a continuous valuation $\tilde{V}:B(K)\longrightarrow \mathbb R$, where $B(K)$ is the space of the bounded Borel functions $g:K\longrightarrow \mathbb R$, endowed with the supremum norm. 

\begin{teo}\label{t:L1factorization}
Let $(\Omega, \Sigma, \mu)$ be  a measurable space with $\mu$ $\sigma$-finite and let $K$ be a compact metrizable space. Let $X$ be either the space of continuous functions $C(K)$ with the supremum norm or an order continuous Banach lattice represented as a function space in $\Omega$ in the sense defined at the beginning of this section. 

A mappping  $V:X\longrightarrow \mathbb R$ is a continuous valuation if and only if 

\begin{enumerate}

\item in the case $X$ is an order continous Banach lattice, there exists a continuous valuation $\Phi:X\longrightarrow L_1(\mu)$, such that, for every $f\in X$, $$V(f)=\int_\Omega \Phi(f) d\mu.$$

\item in case $X=C(K)$, there exists a Borel measure $\mu$ on $K$ and a continuous valuation $\Phi:X\longrightarrow L_1(\mu)$, such that, for every $f\in X$, $$V(f)=\int_K \Phi(f) d\mu.$$

\end{enumerate}

Moreover, in case (1) $\Phi$ can be chosen so that, for every $f\in X$, $A\in \Sigma$, $\Phi(f\chi_A)=\Phi(f)\chi_A$. 

In case (2), $f\chi_A$ need not belong to $C(K)$ in general, but we have that $\Phi$ can be extended to a continuous valuation $\tilde{\Phi}:B(K)\longrightarrow L_1(\mu)$ so that, for every $f\in C(K)$ and every Borel set $A\subset K$,  $\tilde{\Phi}(f\chi_A)=\Phi(f)\chi_A$. 
\end{teo}
\begin{proof}
In both cases, clearly, if $\Phi$ is a continuous valuation then so is $V$. We have to prove the other implication.  
The case (2) is implicit in \cite{TraVi} and explicitly stated in \cite[Proposition 2.2]{TraVi2}. The case (1)  is \cite[Lemma 3.1]{DO69} for the case of finite measure $\mu$. In case $\mu$ is $\sigma$-finite, the result follows from Theorem \ref{t:DOsigmafinitas} defining $\Phi(f)(t)=K(f(t),t)$. 
\end{proof}

\subsection{Valuations invariant under measure preserving transformations}

Let   $(\Omega, \Sigma, \mu)$ be, as before,  a measure space, with $\mu$ $\sigma$-finite and $X$  an order continous Banach lattice represented in $\Omega$. 

In many cases, valuations $V:X\longrightarrow \mathbb R$ invariant under measure preserving transformations are specially interesting. We show next that, in this case, Theorem \ref{t:DOsigmafinitas} admits a much simpler statement. 

First we  need a preliminary lemma. 

\begin{lema}\label{l:proporcionales}
Let $(\Omega, \Sigma)$ be a measurable space. Let $\mu, \nu:\Sigma\longrightarrow \mathbb R$ be two measures defined on it, both of them non-atomic and either both of them finite or both of them $\sigma$-finite.  If $\nu(A)=\nu(A')$ whenever $\mu(A)=\mu(A')$, then there exists $c\in \mathbb R$ such that, for every $A\in \Sigma$, $\nu(A)=c\mu(A)$. 
\end{lema}
\begin{proof} Suppose first that $\mu$ is finite. If $\mu(\Omega)=0$, the result follows easily. Otherwise,  let $c=\frac{\nu(\Omega)}{\mu(\Omega)}$. Since $\mu$ is non-atomic, there exists $A\subset \Omega$ such that $\mu(A)=\frac{\mu(\Omega)}{2}$. Then $\mu(A)=\mu(A^c)=\frac{\mu(\Omega)}{2}$. Therefore $\nu(A)=\nu(A^c)=\frac{\nu(\Omega)}{2}=c\mu(A)$. Similar reasonings show that, for every $a\in \mathbb Q\cap [0,1]$, if $\mu(A)=a\mu(\Omega)$ then $\nu(A)=c\mu(A)$. Let now $A\subset \Omega$ be a set such that $\mu(A)=\alpha\mu(\Omega)$, with  $\alpha \in [0,1]$ an irrational number. Let $(a_n)_{n\in \mathbb N}$ be a  decreasing sequence of rational numbers converging to $\alpha$ and let $(A_n)_{n\in \mathbb N}$ be a decreasing sequence of sets, with $\mu(A_n)=a_n\mu(\Omega)$, converging to $A$. That is, for every $n\in \mathbb N$, $A_{n+1}\subset A_n$ and $\bigcap_n A_n=A$. 

Then $\mu(A_n\setminus A)\rightarrow 0$ as $n$ tends to infinity. It follows easily that also $\nu(A_n\setminus A)\rightarrow 0$. Therefore 
$$\nu(A)=\lim_{n\rightarrow \infty} \nu(A_n)=\lim_{n\rightarrow \infty} c\mu(A_n)=c\mu(A).$$

For the $\sigma$-finite case, note first that there exists an increasing sequence  $(\Omega_n)_{n\in \mathbb N}\subset \Sigma$  such that $\Omega=\bigcup_n \Omega_n$ and $\mu(\Omega_n)<\infty$, $\nu(\Omega_n)<\infty$ for every $n\in \mathbb N$. To see this, just consider one such sequence $(A_n)_{n\in \mathbb N}$ for the measure $\mu$, another one $(B_n)_{n\in \mathbb N}$ for the measure $\nu$ and define $\Omega_n=A_n\cap B_n$. 

Now we use the finite case and we obtain that, for each $n\in \mathbb N$ there exists $c_n\in \mathbb R$ such that $$\nu_{|_{\Omega_n}}=c_n \mu_{|_{\Omega_n}}.$$

Since $\Omega_n$ is an increasing sequence, we have that for every $n, m\in \mathbb N$, $c_n=c_m$, which finishes the proof. 

\end{proof}

Simple examples show that the result is false if we do not require that $\mu, \nu$ are non-atomic, or $\sigma$-finite.

The following observations will be needed next. Let $V:X\longrightarrow \mathbb R$ be a continuous valuation. Given $\lambda\in \mathbb R$, we can define the set function $\nu_\lambda:\Sigma\longrightarrow \mathbb R$ by $ \nu_\lambda(A)=V(\lambda\chi_A)$ when $\chi_A\in X$ and $\nu_\lambda(A)=\infty$ otherwise. Thus defined, $\nu_\lambda$ is finitely additive. It follows from the properties  of $X$ and the continuity of $V$ that $\lim_{\mu(A)\rightarrow 0} \nu_\lambda(A)=0$. Therefore, $\nu_\lambda$ is a countably additive measure, continuous with respect to $\mu$. In case $\mu$ is finite, it follows from  \cite{DO68, DO69} that $K(\lambda, \cdot)$ coincides with the Radon-Nikodym derivative of $\nu_\lambda$ with respecto to $\mu$. Reasoning similarly to the proof of Theorem \ref{t:DOsigmafinitas} is easy to see that this is also true in case $\mu$ is $\sigma$-finite. 

\begin{prop}\label{p:Kconstante}
Let $(\Omega, \Sigma, \mu)$ be a measure space with $\mu$ non-atomic and $\sigma$-finite.  Let  $V:X\longrightarrow \mathbb R$ be a continuous valuation and let $K:\mathbb R\times \Omega \longrightarrow \mathbb R$ be its representing density as in Theorem  \ref{t:DOsigmafinitas}. Then, the following are equivalent: 

\begin{enumerate}

\item[(i)]There exists a continuous function $\theta:\mathbb R\longrightarrow \mathbb R$ such that $K(\lambda, t)=\theta(\lambda)$ for $\mu$-almost every $t\in \Omega$.

\item[(ii)] There exists a continuous function $\theta:\mathbb R\longrightarrow \mathbb R$ such that, for every $A\in \Sigma$, $V(\lambda\chi_A)=\theta(\lambda)\mu(A)$. 
\end{enumerate}

Moreover, if $\mu$ is non-atomic, the next condition is also equivalent to (i) and (ii):

\begin{enumerate}
\item[(iii)] $V$ satisfies that, for every $\lambda\in \mathbb R$, and for every Borel sets $A, A'\subset \Omega$ such that $\mu(A)=\mu(A')<\infty$

$$V(\lambda \chi_A)=V(\lambda \chi_{A'}).$$
\end{enumerate}
\end{prop}

\begin{proof}

(i) implies (ii) follows immediately from the integral representation Theorem \ref{t:DOsigmafinitas}: 

$$V(\lambda\chi_A):=\int_\Omega K(\lambda, t)d\mu(t)=\theta(\lambda)\mu(A). $$

(ii) implies (i): It follows from (i) that, for every $A\in \Sigma  $, $\nu_\lambda(A)=\theta(\lambda)\mu(A)$. In that case, the Radon-Nikodym derivative of $\nu_\lambda$ with respect to $\mu$,  $K(\lambda, \cdot)$ is $\theta(\lambda)$ for $\mu$-almost every $t\in \Omega$. 

(ii) obviously implies (iii), also for atomic $\mu$. 

In case $\mu$ is non-atomic, (iii) implies (i):   Condition (iii) together with  Lemma \ref{l:proporcionales} imply   that $\nu_\lambda=c_\lambda \mu$. Therefore, $K(\lambda, t)=c_\lambda$ for $\mu$-almost every $t\in \Omega$. 

So, for every $\lambda\in \mathbb R$  we define $\theta(\lambda):=c_\lambda$. Continuity of $\theta$ follows inmediately from the continuity of $K$ in the first variable. 

\end{proof}

We can now state the simple form Theorem \ref{t:DOsigmafinitas} takes under this hypothesis.

\begin{corolario}\label{c:representacion}
Let $(\Omega, \Sigma, \mu)$, $X$ be as in  Proposition \ref{p:Kconstante}. If $V:X\longrightarrow \mathbb R$ is a continuous valuation satisfying any of the equivalent conditions of Proposition \ref{p:Kconstante}, then there exists a continuous function $\theta:\mathbb R\longrightarrow \mathbb R$ such that, for every $f\in X$, $$V(f)=\int_\Omega \theta(f(t))d\mu(t).$$

Conversely, if $\theta:\mathbb R\longrightarrow \mathbb R$ is a continuous function such that $\int_\Omega \theta(f(t))d\mu(t)$ exists and is finite for every $f\in X$,  then the mapping $$f\mapsto \int_\Omega \theta(f(t))d\mu(t)$$ defines a continuous valuation $V:X\longrightarrow \mathbb R$ 
\end{corolario}

\begin{proof} The first part of the statement follows immediately from Theorem \ref{t:DOsigmafinitas} and Proposition \ref{p:Kconstante}. The second part follows also from Theorem \ref{t:DOsigmafinitas}. 
 \end{proof}

As interesting particular cases of the above situation, we have, first, the case of translationally invariant valuations defined on $\Omega=\mathbb R^n$ with $\Sigma$ the Borel sets of $\mathbb R^n$ and $\mu$ the Lebesgue measure of $\mathbb R^n$. And, second, we have the case of rotationally invariant valuations defined on $\Omega=\mathcal S^{n-1}$ with $\Sigma$ the Borel sets of $\mathcal S^{n-1}$ and $\mu$ the Haar  measure of $\mathcal S^{n-1}$. We show next that both of them are covered by the previous result.

Suppose first that  $(\Omega, \Sigma, \mu)$ is $\mathbb R^n$ with the Borel $\sigma$-algebra and the Lebesgue measure. Let $X$ be as above and let $V:X\longrightarrow \mathbb R$ be a continuous translationally invariant valuation. Then,  the above mentioned measures $\nu_\lambda$ defined by $\nu_\lambda(A)=V(\lambda \chi_A)$ (whenevery $\mu(A)<\infty$) are translationally invariant and satisfy that $\nu_\lambda(C)=V(\lambda\chi_C)<\infty$, with $C$ the unit cube.  It is well known that, in that case, $\nu_\lambda$ is a constant multiple of the Lebesgue measure, $\nu_\lambda=c_\lambda \mu$. Therefore, $V$ satisfies Condition (ii) in Proposition \ref{p:Kconstante} and, the conclusions of Corollary \ref{c:representacion} apply.

Similarly, suppose that  $(\Omega, \Sigma, \mu)$ is $\mathcal S^{n-1}$ with the Borel $\sigma$-algebra and the Haar measure, and let $X$ be as above. Let $V:X\longrightarrow \mathbb R$ be a continuous rotationally  invariant valuation. Then,  the measures $\nu_\lambda$  are clearly rotationally  invariant and satisfy that $\nu_\lambda(S^{n-1})<\infty$. Then, it is well known that $\nu_\lambda$ is a constant multiple of the Lebesgue measure, $\nu_\lambda=c_\lambda \mu$. Therefore, again $V$ satisfies Condition (ii) in Proposition \ref{p:Kconstante} and the conclusions of Corollary \ref{c:representacion} apply. 

The fact that we are free to choose $X$ among all the order continous Banach lattices supported in $\Omega$ grants the results stated in this section a great generality. Note in particular that they imply essentially all of the main results in \cite{Ts} ($X=L_p$ spaces) and \cite{Ko} ($X$ an Orlicz space) and several of the main results of \cite{Klain97}.

\subsection{Valuations on $L_p$ spaces.}

Due to their special relevance, in this subsection we look with further detail into the valuations defined on $L_p$ spaces. 
In the following $(\Omega, \Sigma, \mu)$ will be, as before,  a $\sigma$-finite measure space and $X=L_p(\mu)$, $1\leq p <\infty$. 

For every continuous valuation $V:L_p(\mu)\longrightarrow \mathbb R$,  Theorem \ref{t:DOsigmafinitas} guarantees the existence of a strong Carath\'eodory function $K:\mathbb R\times \Omega\longrightarrow \mathbb R$ such that, for every $f\in L_p(\mu)$, $$V(f)=\int_\Omega K(f(t), t) d\mu(t).$$

It would be interesting to characterize the strong Carath\'eodory functions $K$  which define valuations on $L_p(\mu)$. For this, $K$ needs to satisfy that the above integral is finite for every $f\in L_p(\mu)$. Clearly, a sufficient condition for this is the existence of $a\in \mathbb R^+$ such that  $|K(\lambda, t)|\leq a |\lambda|^p$ for every $\lambda \in \mathbb R$ and for every $t\in \Omega$ off a set of zero measure. Or, in case $\mu$ is finite, the existence of $a, b\in \mathbb R^+$ such that  $|K(\lambda, t)|\leq a |\lambda|^p+ b$ for every $\lambda \in \mathbb R$ and for every $t\in \Omega$ off a set of zero measure.

It is easy to see that this condition is not necessary: 
Let $\Omega=[0,1]$ and $\mu$ the Lebesgue measure on $[0,1]$. Consider $(A_n)_{n\in \mathbb N}$ a sequence of disjoint Borel sets of $[0,1]$ such that, for every $n\in \mathbb N$,  $\mu(A_n)=2^{-2n}$. Also, for every $n\in \mathbb N$, consider the function $$\varphi_n= \begin{cases}
2^{n-1} x,\, \, & 0\leq x \leq 2\\
2^{n+1}-2^{n-1} x,\, \, & 2< x \leq 4\\
0, \,\, & 4< x
\end{cases}$$

Now, we define $$K:\mathbb R^+ \times [0,1]\longrightarrow \mathbb R$$ by

$$K(\lambda, t)= \begin{cases}
\varphi_n(\lambda), \, \, & t\in A_n\\
0, \,\, & t\not \in \bigcup_n A_n 
\end{cases}$$
Thus defined, $K$ is clearly a strong Carath\'eodory function. It is easy to see that, for every measurable function $f:[0,1]\longrightarrow \mathbb R$, $$\int_{[0,1]} K(f(t), t)d\mu(t)\leq \sum_n\frac{2^n}{2^{2n}}=1.$$

Therefore, the application $V(f)=\int_{[0,1]} K(f(t), t)d\mu(t)$ defines a continuous valuation $V:L_p\longrightarrow \mathbb R$ for every $1\leq p <\infty$. But, clearly, there do not exist $a, b\in \mathbb R$ such that, for every $t$ off a set of zero measure, $K(2, t)\leq a 2^p+ b$ 

\smallskip

In the case of valuations invariant under measure preserving transformations, the situation is simpler and has been already analyzed (\cite{Klain97}, \cite{Ts}). For completeness, we briefly recall here the characterization.

\begin{prop}\label{p:Lp}
Let $(\Omega, \Sigma, \mu)$, be a measure space, with $\mu$ non atomic and $\sigma$-finite. Let $1\leq p <\infty$ and let $V:L_p(\mu)\longrightarrow \mathbb R$ be a continuous valuation satisfying any of the equivalent conditions of Proposition \ref{p:Kconstante}. Let  
 $\theta:\mathbb R\longrightarrow \mathbb R$ be the function representing $V$ as in Corollary \ref{c:representacion}. Then, there exists $a,\,b
\in \mathbb R$ such that,  $$|\theta(\lambda)|\leq a|\lambda|^p+b.$$ (with $b=0$ in case $\mu(\Omega)=\infty$). 

Conversely,  let $a,b \in \mathbb R$ and let $\theta:\mathbb R\longrightarrow \mathbb R$ be a continuous function satisfying that  $|\theta(\lambda)|\leq a|\lambda|^p+b
$ (with $b=0$ if $\mu(\Omega)=\infty$). Then the mapping $$f\mapsto \int_\Omega \theta(f(t))d\mu(t)$$ defines a continuous valuation $V:X\longrightarrow \mathbb R$ which is invariant under measure preserving transformation.  
\end{prop}

\begin{proof}
Let $V, \, \theta$ be as in the hypothesis.  By Theorem \ref{t:jordan} we may assume that $V$ is positive. In that case, $\theta$ is also positive. Suppose first $\mu(\Omega)<\infty$. If the bound is not true. Then, for every $n\in \mathbb N$ there exists $\lambda_n\in \mathbb R$ such that $|\theta(\lambda_n)|> \frac{2^n}{\mu(\Omega)}(|\lambda_n|^p+1)$. Since $\mu$ is not atomic, we can choose a sequence of pairwise disjoint $(A_n)_{n\in \mathbb N}\subset \Sigma$ such that $\mu(A_n)=\frac{\mu(\Omega)}{2^n(|\lambda_n|^p+1)}$. We define the function $f=\sum_n \lambda_n \chi_{A_n}$. Then $f\in L_p(\mu)$, but 
$$
V(f)=\int_\Omega \theta(f(t))d\mu(t)=\sum_n \theta(\lambda_n)\mu(A_n)\geq \sum_n 1\rightarrow \infty,
$$
a contradiction with the fact that $V$ is defined on $L_p(\mu)$. 

Now, if $\mu(\Omega)=\infty$, and the bound for $\theta$ does not hold, then for every $n\in \mathbb N$ we can find $\lambda_n\in \mathbb R$ such that $|\theta(\lambda_n)|>2^n|\lambda_n|^p$. As before, we can now take a sequence of pairwise disjoint $(A_n)_{n\in \mathbb N}\subset \Sigma$ such that $\mu(A_n)=\frac{1}{2^n|\lambda_n|^p}$ and proceed similarly to get a contradiction.

The converse implication follows immediately from the second part of Corollary \ref{c:representacion}. 
\end{proof}

\end{document}